\documentclass[12pt]{amsart} 
\usepackage{amssymb,latexsym, amsmath, amscd, array, graphicx} 

\swapnumbers 
\numberwithin{equation}{section} 

\theoremstyle{plain} 
 
\newtheorem{thm}{Theorem}[section] 
\newtheorem{theorem}[thm]{Theorem}

\newtheorem{prop}[thm]{Proposition} 
\newtheorem{cor}[thm]{Corollary} 
\newtheorem{defteor}[thm]{Definition} 

\newcommand\theoref{Theorem~\ref} 
 
\newcommand\propref{Proposition~\ref} 
\newcommand\corref{Corollary~\ref}

\theoremstyle{definition} 
\newtheorem{defin}[thm]{Definition} 
\newtheorem{definition}[thm]{Definition}

\newtheorem{rem}[thm]{Remark} 
\newtheorem{remark}[thm]{Remark}

\newtheorem{question}[thm]{Question}

\DeclareMathOperator{\Ker}{{\rm Ker}} 
\DeclareMathOperator{\im}{{\rm Im}}

\def\wt{\widetilde}

\def\C{{\mathbb C}} 
 
\def\Q{{\mathbb Q}} 
\def\P{{\mathbb P}} 
\def\R{{\mathbb R}} 
 
\def\T{{\mathbb T}}

\def\rp{{\R \P}}

\def\1{\hbox{\rm\rlap {1}\hskip.03in{\rom I}}} 
\def\Bbbone{{\rm1\mathchoice{\kern-0.25em}{\kern-0.25em} 
{\kern-0.2em}{\kern-0.2em}I}} 

\def\wt{\widetilde}

\def\m{\medskip}

\long\def\forget#1\forgotten{} %

\newcommand\ver[1]{\marginpar{\tiny Changed in Ver \VER}} 

\date{\today} 

\begin{document}  

\title[Essential manifolds]
{Essential manifolds with extra structures} 

\author[S. Kutsak]
{Sergii Kutsak}

\address {Sergii Kutsak, Department of Mathematics, University of Florida,
358 Little Hall, Gainesville, FL 32601, USA}
\email{sergiikutsak@ufl.edu} 
\begin{abstract}
We consider classes of algebraic manifolds  $\mathcal{A}$, of symplectic manifolds $\mathcal{S}$, of symplectic manifolds  with the hard Lefschetz property $\mathcal{HS}$ and the class of cohomologically symplectic manifolds $\mathcal{CS}$. For every class of manifolds $\mathcal{C}$ we denote by $\mathcal{EC}(\pi,n)$ a subclass of $n$-dimensional rationally essential manifolds with fundamental group $\pi$. In this paper we prove that for all the above classes with symplectically aspherical form the condition $\mathcal{EC}(\pi,2n)\ne \emptyset$ implies  that $\mathcal{EC}(\pi,2n-2)\ne\emptyset $ for every $n>2$. Also we prove that all the inclusions $\mathcal{EA}\subset\mathcal{EHS}\subset\mathcal{ES}\subset\mathcal{ECS}$ are proper.
\end{abstract}
\maketitle

\section{Introduction}

  Let $M$  be a closed, connected, orientable manifold of 
dimension $n$ and let $ \pi $ be the fundamental group of $M$. 
Recall that  A map $ f:M \to K(\pi,1) $ is called  a {\em classifying map} for $M$ if $f$ induces an isomorphism $f_* :\pi_1(M, x_0) \to \pi_1(K(\pi,1), f(x_0))$ 
for all $ x_0 \in M$. It is well-known that a classifying map  exists and is unique up to homotopy.
Gromov called a manifold $M$ {\em inessential} if there exists a classifying map $f: M \to K(\pi,1)$ to the $(n-1)$-skeleton
of $K(\pi,1)$. Otherwise $M$ is called {\em essential}~\cite{Gr1}. Gromov noticed that manifolds with positive scalar curvature tend to be inessential. He introduced several classes of essential manifolds (hyperspherical, hypereuclidean, enlargeable,~\cite{Gr2}) for which he jointly with Lawson proved that manifolds of those classes
cannot carry a metric with positive scalar curvature~\cite{GL}.
 The following is found in reference \cite[Lemma 2.4]{DR}.

\begin{prop} An orientable $n$-manifold $ M $ is {\it essential}
if and only if  the homomorphism $ f_\ast : H_n (M) \to H_n (K(\pi,1)) $ induced  by the classifying map $f$ is nontrivial.  Equivalently, if the image of the 
fundamental class $[M]\in H_n(M)$ under $ f_\ast $ is nontrivial in 
the $ n \text{th} $ integral homology group $ H_n( K( \pi,1) )$ of the 
Eilenberg-MacLane space $ K( \pi, 1) $ . 
\end{prop}

For example, the real projective space $ \rp^{2n+1}$ is an essential manifold. 
Every aspherical manifold (for example, the 
torus $\T^n$, a compact orientable surface $M_{g}$ of genus $g$) is essential.  There are no simply connected essential manifolds of positive dimension.

 \begin{definition}
Let $M$  be a closed, connected, orientable manifold of 
dimension $n$ and let $ \pi $ be the fundamental group of $M$. We say that manifold $M$ is {\em rationally essential} if a classifying map $ f:M \to K(\pi,1) $ induces nontrivial homomorphism $f_*:H_n(M;\mathbb{Q}) \to H_n(K(\pi,1);\mathbb{Q})$. 
 \end{definition}

Clearly,  every rationally essential manifold is essential but not vise versa:  $ \rp^{2n+1}$ is not rationally essential.

\m Clearly, if $H_n(K(\pi,1))=0$ then there are no essential (and hence rationally essential) $n$-manifolds with the fundamental group $\pi$. The converse also holds: if $\pi$ is a finitely presented group and $H_n(\pi;\Q)\ne 0$ 
then there exists a rationally essential $n$-manifold with the fundamental group $\pi$, see \theoref{main} below.

\m Brunnbauer and Hanke gave a characterization of Gromov type  classes of rationally
essential manifolds  with given
 fundamental group in terms of group homology~\cite{BH}. In this paper we consider similar problem for some
symplectic type classes.

\m
Given a class of manifolds $\mathcal C$ we denote by  $\mathcal{EC}$ the subclass that consists of rationally essential manifolds.
Here we consider the following classes:
$$
\mathcal{A}\subset\mathcal{HS}\subset\mathcal{S}\subset\mathcal{CS} 
$$ 
where
$\mathcal{A}$ is the class of algebraic manifolds,
$\mathcal{S}$  is the class of symplectic manifolds,
$\mathcal{HS}$  is the class of symplectic manifolds with the hard Lefschetz property,
and
$\mathcal{CS}$  is the class of cohomologically symplectic manifolds (see sections \ref{s3} and \ref{s4} below).
It is known that all the above inclusions of classes are proper \cite{C,TO,G1,DGMS}. We will show that the inclusions of the essential counterparts are also proper. \\

For every class of manifolds $\mathcal C$ we denote by $\mathcal C(\pi,n)$ a subclass of $n$-dimensional manifolds
with fundamental group $\pi$. This paper is an attempt to study the following question.

\

MAIN PROBLEM. {\em  For which values $\pi$ and $n$, is $\mathcal{EC}(\pi, n)$ non-empty?}

\m 
In particular, in the paper we address the following conjecture proposed by Dranishnikov and Rudyak:

\

CONJECTURE. {\em For the first three above classes for $n>2$ the condition $\mathcal{EC}(\pi, 2n)\ne\emptyset$ implies  that 
$\mathcal{EC}(\pi,2n-2)\ne\emptyset $}.

\m
We prove for all the above classes a weaker version of the conjecture that deals with
symplectically aspherical manifolds, see Section 3 for the definition. 

\m
Note that every complex projective algebraic manifold $V$ is symplectic: the corresponding symplectic form is given by the K\"ahler form ~\cite[page 109]{GH}. In particular, we are able to speak about symplectically aspherical algebraic manifolds.

\section{Acknowledgements}
I would like to thank my advisors Alexander Dranishnikov and Yuli Rudyak for their support and guidance throughout this work.
I am very grateful for their insightful discussions and useful advices.

\section{Preliminaries}\label{s3}
The following fact is known (see for example {\cite{BH},\cite{Dr}). 
Since there is no detailed proof of it in print, we give a complete proof here.
\begin{theorem}\label{main} For every finitely presented group $ \pi $ and every integer $ n $ if $ H_n(\pi;\mathbb{Q}) \neq 0 $ then for every nontrivial element $ \alpha \in  H_n(\pi;\mathbb{Q}) $  there exists a closed, connected, orientable $n$-manifold $M$, an integer $ k \neq 0 $ and a map $ f:M \to K(\pi,1) $ such that $ f_*([M])= k\alpha $ and $ f_*:\pi_1(M) \to \pi_1(K(\pi,1)) $ is a group isomorphism. 
\end{theorem}

\begin{proof}
 Let $ \pi $ be a finitely presented group and let $ n  $ be an integer such that $ H_n(\pi;\mathbb{Q}) \neq 0 $. Take any nontrivial element $ \alpha $ in
  $ H_n(\pi;\mathbb{Q}) $. Because of a theorem of Thom, there exist  a closed $n$-manifold $ N $, an integer $ k \neq 0 $ and a map $ g:N \to K(\pi,1)$ such that $g_*([N])= k\alpha $, see e.g.~\cite[Theorem IV.7.36]{R}. Suppose that $ g_*:\pi_1(N) \to \pi_1(K(\pi,1)) $ is not surjective. Let $ \alpha:[0,1] \to K(\pi,1) $ be a loop such that $ [\alpha] \in \pi_1(K(\pi,1))\backslash im(g_*) $  and 
$ \alpha(0)=\alpha(1)=y_0 $. Without loss of generality we can assume that $ y_0 \in \im (g) $ since the fundamental groups of $ K(\pi,1) $ based at different points are isomorphic because $ K(\pi,1) $ is path connected. Take $ x_0 \in N $ such that $ f(x_0)=y_0 $. Consider chart $ (U,\varphi) $ on 
$ N $ such that $ \varphi(U)=\mathbb{R}^n $ and $ \varphi(x_0)=0 $. Now define function $ h:\mathbb{R}^n \to \mathbb{R}^n $ in generalized spherical coordinates as follows 
\[ 
h(r,\theta_1,...,\theta_{n-1})=\left \{\begin{array}{cl}
0, &\text{ if }  0 \leq r \leq 1,\\
(r-1,\theta_1,...,\theta_{n-1}),&\text{ if }  r>1.
\end{array}
\right.
\] 
To perform a surgery on a manifold $ N $ we shall define a new function $ \wt{g} $ by: $ \wt{g}(x)=g(x) $ if $ x \notin U $, 
$ \wt{g}(x)=g(\varphi^{-1}h\varphi(x)) $ if $ x \in U $. Then $ \wt{g} $ is homotopic to $ g $ because $ h $ is homotopic to the identity map on $ \mathbb{R}^n $. Let $ D $ be the preimage under $ \varphi $ of the unit ball in $ \mathbb{R}^n $ centered at $ 0 $. Now we perform a surgery on the manifold $ N $. There exists an embedding $ i: S^0 \times D^n \to N $ such that $ i(S^0 \times D^n) \subseteq D $ and $ x_0 \notin i(S^0 \times D^n) $. Form a new manifold from the union of $ N \times I $ and $ D^1 \times D^n $ by attaching $ S^0 \times D^n $ to its image under $ i \times 1 $. We can extend map $\wt{g} \times 1$ by defining $ \wt{g} $ on $ D^1 \times D^n $ as follows 
\[ 
 \wt{g}(t,x)=\alpha(t)  \text{ for all }  (t,x) \in D^1 \times D^{n-1}.  
\]
Connect point $ x_0 $ with points $ (0,c),(1,c) $ in $ D^1 \times D^{n-1} $ for some $ c \in D^n $ by paths $ \gamma_1(t), \gamma_2(t) $ respectively. Let $ \beta(t)=(t,c) \in D^1 \times D^n $ for all $ t \in [0,1] $. Then $ (\wt{g} \times 1)_*(\gamma_1\beta\gamma_2^{-1})=\alpha $.
So we can construct a manifold $ \wt{N} $ and a map $ \wt{g}:\wt{N} \to K(\pi,1) $ such that $ \wt{g_*}([\wt{N}])=k\alpha $ and $ \wt g_* $ induces an epimorphism on fundamental groups. Now we want to perform surgeries that annihilate the elements that generate the kernel of $\wt{g_*}$. Note that since $\wt{N}$ is orientable then every loop $\gamma$ in $\wt{N}$ can be homotoped to a loop $\wt{\gamma}$ that has
trivial normal bundle in $\wt{N}$. Clearly, if a loop $ \wt{\gamma}$ is trivial then the loop $\alpha \wt{\gamma} \alpha^{-1}$ is also trivial
for every path $\alpha:[0,1] \to \wt{N}$ such that $\alpha(1)=\wt{\gamma}(0)$. Since $\Ker(\wt{g_*})$ is normally finitely generated ~\cite{W} then  we can perform surgery on $\wt{N}$ finitely many times to construct a manifold $M$ and a map $f:M \to K(\pi,1)$ that induces isomorphism $f_*:\pi_1(M) \to \pi_1(K(\pi,1))$ and such that $f_*([M])=k\alpha$.
\end{proof}

Note that every oriented manifold of dimension $\le 2$ is essential, an oriented 3-manifold $M$ is essential iff the group $\pi_1(M)$ is not free, \cite {GG, RO}.   

\begin{defin} We define a cohomology class $v\in H^m(X;G)$ to be {\it aspherical} if $v=f^*(a)$ for a classifying map $f: X\to K(\pi_1(X),1)$ and some $a\in H^m(K(\pi_1(X),1);G)$.
\end{defin}

Note that if a class $v$ is aspherical and $v^k\ne 0$ then $v^k$ is aspherical.

\begin{prop}\label{p:aspher} Let $M$ be a closed, orientable manifold of dimension $km$, and let $u\in H^m(M;\Q)$ be an aspherical class. If $u^k\ne 0$, then $M$ is rationally essential.
\end{prop}

 \begin{definition} 
 A {\em symplectic structure} on a smooth manifold $ M $ is a non-degenerate skew-symmetric closed 2-form $ \omega \in \Omega^2(M) $. A {\em symplectic manifold} is a pair $ (M,\omega) $ where $ M $ is a smooth manifold and $ \omega $ is a symplectic structure on $ M $. 
\end{definition}

The non-degeneracy condition means that for all $ p \in M $ we have the property that if $ \omega(v,w)=0 $ for all $ w \in T_p M $ then $ v=0 $. The skew-symmetry condition means that for all $ p \in M $ we have $ \omega(v,w)=-\omega(w,v) $ for all $ v,w \in T_p M $. The closed condition means that the exterior derivative $ d\omega $ of $ \omega $ is identically zero.  Since each odd-dimensional skew-symmetric matrix is singular, we see that $M$ has even dimension. Every symplectic $2n$-dimensional manifold $ (M,\omega) $ is orientable since the $n$-fold wedge product $ \omega \wedge ...\wedge \omega $ never vanishes. 

\begin{definition} A symplectic manifold $(M,\omega)$ is {\em symplectically aspherical} if
 \begin{equation*}
 \displaystyle{\int_{S^2}{f^*\omega}=0   }
 \end{equation*} 
for every smooth map $ f: S^2 \to M $.
 \end{definition}  

\m Clearly, if $ \pi_2(M)=0 $  then a symplectic manifold $ (M,\omega) $ is symplectically aspherical. However there are symplectically aspherical manifolds with nontrivial $\pi_2$,~\cite{G2, IKRT}.

\begin{rem} The cohomology class $[\omega]$ in a  symplectically aspherical manifold $(M,\omega)$ is aspherical. It follows from classical results of Hopf,~\cite{H}(see also ~\cite[Theorem 8.17]{CLOT},~\cite[Theorem 5.2]{B}).  
\end{rem}

In view of this remark and \propref{p:aspher}, we have the following corollary

\begin{cor}\label{c:sas} Every closed symplectically aspherical manifold is rationally essential.
\end{cor}

To proceed, we need the following theorems, see e.g~\cite[page 41]{M}. 

\begin{theorem}
[Lefschetz Hyperplane Theorem]
\label{t:Lefschetz} 
Let $V$  be a complex projective algebraic variety of complex dimension k which lies 
in the complex projective space $ \mathbb{C}P^n $, and let $ P $ 
be a hyperplane in $ \mathbb{C}P^n $ which contains the singular 
points $($if any$)$ of $V$. Then the relative homotopy groups 
$ \pi_r(V, V \cap P ) $ are equal to zero for all $ r<k $.  
\end{theorem}

Note that $ V \cap P $ is a manifold (i.e. non-singular variety) if $V$ is.

\begin{theorem}[Donaldson~\cite{D}]\label{t:D1}
Let $ L \to V $ be a complex line bundle over a compact symplectic manifold $ (V,\omega) $ with compatible almost-complex structure, and with the first Chern class $\displaystyle{c_1(L)=\Bigg[\frac{\omega}{2\pi}\Bigg] }$. Then there is a constant $ C $ such that for all large $ k $ there is a section $ s $ of $ L^{\otimes k } $ with 
\begin{equation}\label{eq:1}
\displaystyle{|\bar{\partial}s|<\frac{C}{\sqrt{k}}|\partial s|}
\end{equation}
on the zero set of $ s $.
\end{theorem}

\begin{theorem}[Donaldson~\cite{D}]\label{t:D2}
Let $ W_k $ be the zero-set of a section $ s $ of $ L^{\otimes k } \to V $ satisfying the conditions of \theoref{t:D1}.
When  $ k $ is sufficiently large the inclusion $ i:W_k \to V $ induces an isomorphism on homotopy groups $ \pi_{p} $ for 
$ p \leq n-2 $ and a surjection on $ \pi_{n-1} $.
\end{theorem}

In view of \theoref{t:D1} and \theoref{t:D2} we obtain the following corollary

\begin{cor}\label{c:sym}
Let $(M, \omega)$ be a closed symplectic manifold of dimension ${2n}$ such that the cohomology class $[\omega]$ is integral.   Then there exists a symplectic  submanifold $V$ of $M$ of codimension 2 such that inclusion $ i: V \to M$ induces an isomorphism on homotopy groups $ \pi_p $ for $ p \leq n-2 $ and a surjection on $ \pi_{n-1}$. Furthermore, the homology class $[V]$ in $M$ is the Poincar\'e dual to a class $r[\omega]$ for some integer $r$.
\end{cor}

\begin{proof} The proof follows from \theoref{t:D1} and \theoref{t:D2} with $\omega$ normalized such that $c_1(L)=[\omega]$. Let $V$ be the zero-set of a section $s$ of $L^{\otimes k} \to M $ as in 
\theoref{t:D1}. Then inequality  \eqref{eq:1} guarantees the existence of symplectic structure on $V$. So $V$ is a symplectic submanifold  of $M$ of codimension 2. The homology class of $V$ is Poincar\'e dual to the first Chern class of $L^{\otimes k}$ up to a multiplicative constant $r$. Finally, according to \theoref{t:D2} the inclusion $ i:V \to M $ induces an isomorphism on homotopy groups  $\pi_p$ for $p\leq n-2$ and a surjection on $\pi_{n-1}$.
\end{proof}
 
\section{Classes of essential manifolds}{\label{s4}

 \begin{theorem}\label{t:alg}
 Assume that $ M $ is a complex projective algebraic manifold of $($real$)$  dimension $2k$  which lies in the complex projective space $ \C P^{N} $. Suppose also that $M$ is symplectically aspherical. Then for every integer $m$ with
 $ 2 \leq m \leq k $ there exists a rationally essential algebraic manifold $ V$  of  dimension $ 2m $ with fundamental group  isomorphic to $\pi_1(M)$.
 \end{theorem}

 \begin{proof}
The case $m=k$ is the \corref{c:sas}.  By induction, it suffices to prove the theorem for $m=k-1$. Indeed, assume that $\dim M=2k>4$ and let $ V= M \cap \C P^{N-1}$. If we prove that $V$ is a rationally essential complex algebraic manifold with $\dim V=2k-2>4$ and the fundamental group $\pi=\pi_1(M)$, we apply the previous argument for $V$ instead of $M$. Because of the \theoref{t:Lefschetz}, $ \pi_r(M,V)=0$ for $ r<k-1$. From the exactness of the homotopy sequence 
 \[ 
\pi_2(M,V) \to \pi_1(V) \to \pi_1(M) \to \pi_1(M,V)
\]
 it follows that 
\[ 
\pi_1(M) \simeq \pi_1(V) \simeq \pi  \text{ since }   \pi_2(M,V) \simeq \pi_1(M,V)\simeq 0. 
\]
 Hence $V$ is a complex algebraic manifold with fundamental group isomorphic to $ \pi $, and $\dim V=\dim M-2$. 
 \m
It remains to prove that $V$ is rationally essential. But this follows from \corref{c:sas} because the induced K\"ahler form on $V$ is aspherical. 
 \end{proof}  

 \begin{theorem}\label{t:sympl} Let $ (M,\omega) $ be a closed symplectically aspherical manifold of dimension $2n > 2 $ with fundamental group $ \pi $.
   Then for every $k$ such that $ 2 \le k \le n $ there exists a symplectically aspherical manifold $V$ of  dimension $2k$ with fundamental group isomorphic to $ \pi $. \end{theorem}

 \begin{proof} We prove the theorem by induction. Similarly to the proof of \theoref{t:alg}, it suffices to prove the case $k=n-1$. Without loss of generality, we can assume that the cohomology class $[\omega]$ is integral (see \cite[Prop. 1.5]{IKRT}). Let $M$ be a manifold as in \corref{c:sym}. Then, for $n>2$, the inclusion $ i:V\to M$ induces an isomorphism on the fundamental groups $ \pi_1(V) \to \pi_1(M) $.  Now, $V$ is a symplectic manifold with symplectic structure $ i^*\omega $ induced from $M$. It is clear that 
$$ 
\int_{S^2}{g^*i^*\omega}=0 
$$ 
for every map $ g:S^2 \to V$. Thus $ (V,i^*\omega) $ is a symplectically aspherical manifold of dimension $2n-2$ with $\pi_1(V)=\pi$. 
\end{proof}

 \begin{definition} A symplectic manifold $ (M^{2n},\omega)$ has the {\em hard Lefschetz property} (HLP) if the map 
$$
 L^k_{[\omega]}:H^{n-k}_{DR}(M^{2n}) \to H^{n+k}_{DR}(M^{2n}), \qquad
L^k_{[\omega]}([x])=[\omega^k \wedge x] 
$$
 is an isomorphism for all $ k=0,\dots,n $.
 \end{definition}

 For example, the Hard Lefschetz Theorem says that every K\"ahler manifold has HLP, see~\cite[page 122]{GH}.

 \begin {theorem}\label{t:hard}
Let $ (M,\omega) $ be a  symplectically aspherical manifold of dimension $2n>2$ with fundamental group $ \pi $ and having HLP. Then for every $ m $ such that $ 2 \leq m \leq n $ there exists a symplectically aspherical manifold $ (V,\eta) $ of dimension $ 2m $ with fundamental group isomorphic to $ \pi $ and having HLP.
 \end{theorem}
   
 \begin{proof}
We follow the proof of \theoref{t:sympl} and must prove that the manifold $V$ as in \theoref{t:sympl}  has HLP.

First, we need to show that $ L^k_{[\omega^*]}:H^{n-1-k}(V) \to H^{n-1+k}(V) $ is an isomorphism for all $ k=0,\dots,n-1 $ where $ \omega^* $ is the pullback of $ \omega $ under inclusion $ i:V \to M $. We need to consider separately the case when $k=0$. So fix any $ k $ such that $ 0 < k \leq n-1 $.  Since
$H^{n-1-k}(V)$ and $H^{n-1+k}(V) $ have the same dimension, it suffices to show that $ L^k_{[\omega*]} $ is a monomorphism. Consider the  following commutative diagram
$$
\CD
 H^{n-1-k}(M) @>L^k_{[\omega]}>> H^{n-1+k}(M) @>{\smile \omega}>> H^{n+1+k}(M)\\
@V i_1^* VV    @VV i_2^* V   @.\\
H^{n-1-k}(V)@>L^k_{[\omega^*]}>> H^{n-1+k}(V) @.
\endCD
$$
where $ L^k_{[\omega]} $ is a monomorphism because $ L^{k+1}_{[\omega]} $ is
 an isomorphism. It follows from \corref{c:sym}, and Whitehead theorem~(see \cite[page 399]{S}) that $i_1^*$ is an isomorphism.
 Hence it suffices to show that $ i_2^* $ is a monomorphism on the $ \im(L^k_{[\omega]}) $. Assume that $ \alpha \in H^{n-1-k}(M) $ is nontrivial and $ i_2^*(\alpha \smile \omega^k)=0 $. Then 
\begin{equation*}
\begin{aligned}
 0 \neq & r([M] \frown(\alpha \smile \omega^{k+1}))=
 r([M] \frown(\alpha \smile \omega^k)) \frown \omega\\
 = & r([M] \frown \omega) \frown (\alpha \smile \omega^k)=
 i_*([V]) \frown ( \alpha \smile \omega^k)\\
= & i_*([V] \frown i_2^*(\alpha \smile \omega^k))=0. 
\end{aligned}
\end{equation*}
This is a contradiction. So  $ L^k_{[\omega^*]} $ is an isomorphism for all $ k=1,\dots,n-1 $. 

If $k=0$ then it is obvious that $L^0_{[\omega^*]}:H^{n-1}(V) \to H^{n-1}(V)$ is an isomorphism. Thus $ V $ is a  symplectically aspherical manifold of dimension $ 2n-2 $ with fundamental group $ \pi $ having the HLP. 

\m Now we can apply the above procedure to $V$, and the result follows by induction. 
 \end{proof}

\begin{defteor}[Lupton-Oprea~\cite{LO}]\rm
A  manifold $ M $ of dimension $ 2n $ is {\em cohomologically symplectic} (or, briefly, c-symplectic) if there exists a closed differential 2-form $ \omega $ on $ M $ such that $ [\omega]^n \neq 0$ .
\end{defteor}

Clearly, not all c-symplectic manifolds are symplectic. For example, $ \C P^2\#\C P^2 $ is c-symplectic but is not symplectic \cite{G1}. 

\begin{theorem}
Let $ (M,\omega) $ be a  c-symplectic manifold of dimension $2n>2$ with fundamental group $ \pi $ and with aspherical c-symplectic form. Then for every $ m $ such that $ 2 \leq m \leq n $ there exists a c-symplectic
manifold $ (V,\eta) $ of dimension $ 2m $ with fundamental group isomorphic to $ \pi $ and  with aspherical c-symplectic form. 
\end{theorem}

\begin{proof}
Let $ f:M \to K(\pi,1) $ be a classifying map for $ M $. Then
$ \omega=f^*a $ for some $ a \in H^2(K(\pi,1)) $.   There exists  a $ (2n-2) $-dimensional submanifold $ N $ of $ M $ such that $ [N]=r\eta $ for some $ r \in \mathbb{Z} $, where $ \eta=PD([\omega])=[M] \frown \omega $.   Let 
$ i:N \to M $ be the inclusion of $ N $ into $ M $. We want to show that $ (i^*\omega)^{n-1}\neq 0 $. Suppose that $ (i^*\omega)^{n-1}= 0 $. Then
\begin{equation*}
\begin{aligned}
 0 \neq & r([M] \frown \omega^n)=r([M]\frown \omega \frown \omega^{n-1})=\\
 = & i_*([N]) \frown \omega^{n-1}=i_*([N] \frown (i^*\omega)^{n-1})=0 .
\end{aligned}
\end{equation*}
This is a contradiction. Hence $ (i^*\omega)^{n-1} \neq 0 $.
By using surgery we can construct a manifold $ N' $ and a map $ i':N' \to M $ that induces an isomorphism  on the fundamental groups. Moreover,  there exist a manifold $ W $ with $\partial W=N \sqcup N' $ and a map $ g:W \to M $ that extends $ i $ and $ i' $.  In other words, the singular manifolds $i: N \to M$ and $ i': N' \to M$ are bordant:
\[\renewcommand{\arraystretch}{2}
\begin{array}{ccccc}
N &\xrightarrow{\text{$j$}}&W&\xleftarrow{\text{$j'$}}&N'\\
&\searrow{\text{$i$}}&\Big\downarrow{\text{$g$}}&\swarrow{\text{$i'$}}\\
&&M
\end{array}\]
where $j$ and $j'$ are the inclusions.   Thus $ i'_*([N'])=i_*([N]) $.  Now
\begin{equation*}
\begin{aligned}
 \langle(i'^*\omega)^{n-1},[N']\rangle=\langle\omega^{n-1},i'_*([N'])\rangle=\\
 = \langle\omega^{n-1},i_*([N]\rangle=\langle(i^*\omega)^{n-1},[N]\rangle\neq 0 ,
\end{aligned}
\end{equation*}
so $ (i'^*\omega)^{n-1} \neq 0 $.
Thus $ (N',i'^*\omega) $ is a c-symplectic manifold of dimension $ 2n-2 $ with fundamental group isomorphic to $ \pi $. Clearly, $ i'^*\omega $ is an aspherical form because
$ i'^*\omega=(f\circ i')^*a $. The result follows by induction.
\end{proof}

\begin{prop}\label{p:csym}
There is an example of a rationally essential 4-dimensional c-symplectic manifold $M$ which is not symplectic.
\end{prop}
\begin{proof}
Let $\Sigma$ be an aspherical 4-dimensional homology sphere (see \cite{RaT}).
We consider the connected sum $M=\C P^2\#\C P^2\#\Sigma$ and show that it does not admit an almost complex structure.
According to the result of Ehresmann and Wu, a compact 4-manifold $ M $ has an almost complex structure with first Chern class $ c_1 \in H^2(M,\mathbb{Z}) $ if and only if $ c_1 $ reduces modulo 2 to the second Stiefel-Whitney class $ w_2 $ and 
$$ 
c_1^2([M])=3\tau+2\chi, 
$$
where $ \chi $ is the Euler characterictic of $ M $ and $ \tau $ is its signature~(\cite[page 119]{MS}). A routine computation shows that $ \chi=4 $, $ \tau=2 $ and
$ c_1^2([M]) $ is the sum of squares of two integers. But 14 can not be represented in such form. Hence $ M $ does not admit an almost complex structure and therefore is not a symplectic manifold because every symplectic manifold admits a compatible almost complex structure. Furthermore, $ \Sigma =K(\pi_1(\Sigma),1) $, and the collapsing map $ f:M \to \Sigma $ has degree 1. Thus $ M $ is a rationally essential manifold since the homomorphism induced by $ f $ on the 4th homology groups $ f_*:H_4(M;\mathbb{Q}) \to H_4(\Sigma;\mathbb{Q}) $ is nontrivial. 

Since $\Sigma$ is a homology sphere, the collapsing map $ i:M \to \mathbb{C}P^2\#\mathbb{C}P^2 $ induces the isomorphism
$$ i^*:H^2(\mathbb{C}P^2;\mathbb{R}) \oplus H^2(\mathbb{C}P^2;\mathbb{R}) \to H^2(M;\mathbb{R}).  $$ 
Let $ \{[\omega_1],[\omega_2]\}$ be a basis of $ H^2(\mathbb{C}P^2;\mathbb{R})\oplus H^2(\mathbb{C}P^2;\mathbb{R}) $. Then $ i^*([\omega_1]+[\omega_2])^2 \neq 0 $ in $ H^4(M;\mathbb{R}) $. Hence $ M $ is a c-symplectic manifold. 
\end{proof}
\begin{remark}
Note that the Dranishnikov-Rudyak conjecture is not true for c-symplectic manifolds. Consider a rationally essential c-symplectic  manifold $ M=\C P^4\#\C P^4\#(\Sigma \times \Sigma) $ with fundamental group $ \pi_1(M) \simeq \pi_1(\Sigma) \times \pi_1(\Sigma) $. Since $ \Sigma \times \Sigma $ is the Eilenberg-MacLane space $ K(\pi_1(M),1) $ and $ H_6(\Sigma \times \Sigma;\Q) $ is trivial then there does not exist a rationally essential 6-manifold with fundamental group isomorphic to $ \pi_1(M) $.
\end{remark}
\begin{theorem}\label{t:proper}
All the inclusions of classes
$$
\mathcal{EA}\subset\mathcal{EHS}\subset\mathcal{ES}\subset\mathcal{ECS} 
$$ are proper.
\end{theorem}
\begin{proof} 
 First we prove that the inclusion $ \mathcal{EA}\subset\mathcal{EHS} $ is proper. Let $ \mathbb{H} $ be
the Heisenberg manifold.  Then the blow-up $ M $ of $ \mathbb{H} \times \mathbb{H} $ along a torus is a symplectic manifold that satisfies the hard Lefschetz property and has nontrivial triple Massey product~\cite{C}. Since $ \mathbb{H} $ is an aspherical manifold then $ \mathbb{H} \times \mathbb{H} $ is the Eilenberg-MacLane space. So $ M $ is a rationally essential manifold because there exists a degree 1 (classifying) map $ f:M \to \mathbb{H} \times \mathbb{H} $. Note that $M$ is not algebraic since it has non-trivial Massey product, while all K\"ahler (and therefore algebraic)
manifolds are formal spaces, \cite{DGMS}, and hence all their Massey products are trivial.

\m Now we prove that the inclusion $ \mathcal{EHS}\subset\mathcal{ES} $ is proper. Consider the Kodaira-Thurston manifold KT 
obtained by taking the product of the Heisenberg manifold $ \mathbb{H} $ and the circle $ S^1 $. It is well-known that KT is a symplectic manifold. The Kodaira-Thurston manifold is rationally essential because it is a nilmanifold and it can not have the hard Lefschetz property because a symplectic nilmanifold of Lefschetz type is diffeomorphic to a torus~\cite{BG}. 

\m We have already shown that the inclusion $ \mathcal{ES}\subset\mathcal{ECS} $ is  proper, see \propref{p:csym} above.

\end{proof}

The Dranishnikov-Rudyak conjecture cannot be reduced to the aspherical case in view of the following
\begin{prop}
The blow up of a 4-torus  at a single point $M=T^4\#\overline{\C P^2}$ is an algebraic manifold which does not admit an
aspherical symplectic form.
\end{prop}
\begin{proof}
Let $ \omega $ be a symplectic form on $ M $. Then $ \displaystyle{\int_M{\omega^2} \neq 0 } $. We can obtain a form $ \omega' $ on $ \overline{\C P^2} $ that extends the restriction of $ \omega $ on $ \overline{\C P^2}\setminus D $ such that $ \displaystyle{ \int_{\overline{\C P^2}}{\omega'^2}\neq 0} $ where $ D $ is a small enough disk. Then there exists a map $ f:S^2\to \overline{\C P^2}\setminus D $ with
$ \displaystyle{\int_{S^2}{f^*\omega'} \neq 0} $ because if we assume that $ \displaystyle{\int_{S^2}{f^*\omega'}}=0 $ for all maps 
$ f:S^2 \to \overline{\C P^2}\setminus D $ then $ [\omega'] =0 $ in $ H^2(\overline{\C P^2};\R) $. Therefore $ [\omega']^2=0 $ and 
$ \displaystyle{\int_{\overline{\C P^2}}{\omega'^2}=0} $ which contradicts to the choice of $ \omega' $. Consider $f:S^2 \to \overline{\C P^2}\setminus D$ such that $\displaystyle{\int_{S^2}{f^*\omega'}\neq 0}.$ Since $ \omega $ and $ \omega' $ coincide on $ \overline{\C P^2}\setminus D $ then $ \displaystyle{\int_{S^2}{f^*\omega} \neq 0 } $. Thus $ \omega $ is not an aspherical symplectic form.
\end{proof} 

It is natural to consider the class of K\"ahler manifolds $\mathcal{K}$ and ask whether the inclusions $\mathcal{EA}\subset\mathcal{EK}\subset\mathcal{ES}$ are proper. It is known that inclusions $\mathcal{A}\subset\mathcal{K}\subset\mathcal{S}$ are proper \cite{V},\cite{C} and manifold $M$ in \theoref{t:proper} shows that inclusion $\mathcal{EK}\subset\mathcal{ES}$ is also proper. Note that $M$ is not K\"ahler because it is not formal.

\begin{question}
Does there exist an essential K\"ahler manifold that is not algebraic?
\end{question}

\begin{question}
In view of the theorems proved above we may ask whether the Dranishnikov-Rudyak conjecture holds true for the class of K\"ahler manifolds
with aspherical K\"ahler form.
\end{question}

 \end{document}